\documentclass[12pt]{amsart}
\theoremstyle{plain} \numberwithin{equation}{section}
\newtheorem{Theorem}{Theorem}
\newtheorem{Lemma}[Theorem]{Lemma}

\newtheorem{Proposition}[Theorem]{Proposition}

\newtheorem{Corollary}[Theorem]{Corollary}

\theoremstyle{remark}

\title[Riesz projections]
{Deviations of Riesz projections of Hill operators with singular
potentials}

\author{Plamen Djakov}

\author{Boris Mityagin}

\begin{document}

\address{Sabanci University, Orhanli,
34956 Tuzla, Istanbul, Turkey}

 \email{djakov@sabanciuniv.edu}

\address{Department of Mathematics,
The Ohio State University,
 231 West 18th Ave,
Columbus, OH 43210, USA} \email{mityagin.1@osu.edu}

\begin{abstract}
It is shown that the deviations $P_n -P_n^0$ of Riesz projections
$$
 P_n = \frac{1}{2\pi i} \int_{C_n}
(z-L)^{-1} dz, \quad C_n=\{|z-n^2|= n\},
$$
of Hill operators $L y = - y^{\prime \prime} + v(x) y, \; x \in
[0,\pi],$ with zero and $H^{-1}$ periodic potentials go to zero as
$n \to \infty $ even if we consider $P_n -P_n^0$ as operators from
$L^1$ to $L^\infty. $  This implies that all $L^p$-norms are
uniformly equivalent on the Riesz subspaces $Ran \,P_n. $
\end{abstract}

\maketitle

\section{Introduction}

 We consider the Hill operator
\begin{equation}
\label{01} Ly = - y^{\prime \prime} + v(x) y, \qquad x \in
I=[0,\pi],
\end{equation}
with a singular periodic potential $v, \; v(x+\pi) = v(x), \;v \in
H^{-1}_{loc} (\mathbb{R}),$ i.e., $$ v(x) = v_0 + Q^\prime (x),$$
where $$Q \in L^2_{loc}(\mathbb{R}), \quad Q(x+\pi) = Q(x), \quad
w(0) = \int_0^\pi Q(x) dx =0, $$ so $$ Q = \sum_{m \in
2\mathbb{Z}\setminus \{0\}} w (m) e^{ imx},\quad \|v|H^{-1} \|^2 =
|v_0|^2 + \sum_{m \in 2\mathbb{Z}\setminus \{0\}} |w(m)|^2/m^2
<\infty. $$

A. Savchuk and A. Shkalikov \cite{SS03} gave thorough spectral
analysis of such operators.  In particular, they consider a broad
class of boundary conditions (bc) -- see (1.6), Theorem 1.5 there --
in terms of a function $y$ and its quasi--derivative
$$
u = y^\prime - Q y.
$$
Now the natural form of periodic or antiperiodic $(Per^\pm)$ bc is
the following one:
\begin{equation}
\label{02}
 Per^\pm: \quad y(\pi) = \pm y(0), \quad u(\pi) = \pm u(0)
\end{equation}
If the potential $v$ happens to be an $L^2$-function these $bc$ are
identical to the classical ones (see discussion in \cite{DM16},
Section 6.2).

 The Dirichlet bc is more simple:
$$ Dir: \quad  y(0) =0,  \quad y(\pi) =0; $$
it does not require quasi--derivatives, so it is defined in the same
way as for $L^2$--potentials $v$.

In our analysis of instability zones of Hill and Dirac operators
(see \cite{DM15} and the comments there) we follow an approach
(\cite{KM1,KM2,DM3,DM5,DM7,DM6}) based on Fourier Method. But in the
case of singular potentials it may happen that the functions
$$
u_k = e^{ikx} \quad \mbox{or} \quad \sin kx, \;\;k \in \mathbb{Z},
$$
have their $L$--images outside $L^2.$ Moreover, for some singular
potentials $v$ we have $Lf \not \in L^2$ for {\em any smooth} (say
$C^2 -$) nonzero function $f.$ (For example, choose
$$
v(x) = \sum_{r} a(r) \delta_* (x-r),\quad  r  \; \mbox{rational},
r\in I,
$$
with $a(r) >0, \; \sum_{r} a(r) = 1 $ and $\delta_* (x) = \sum_{k
\in \mathbb{Z}} \delta (x-k \pi).$)

This implies, for any reasonable bc, that the eigenfunctions
$\{u_k\}$ of the free operator $L^0_{bc}$ are not necessarily in the
domain of $L_{bc}.$  Yet, in \cite{DM17,DM16} we gave a
justification of the Fourier method for operators $L_{bc}$ with
$H^{-1}$--potentials and $bc = Per^\pm $ or $Dir.$ Our results are
announced in \cite{DM17}, and in \cite{DM16} all technical details
of justification of the Fourier method are provided.

Now, in the case of singular potentials, we want to compare the
Riesz projections $P_n$ of the operator $L_{bc}, $ defined for large
enough $n$ by the formula
\begin{equation}
\label{03}
 P_n = \frac{1}{2\pi i} \int_{C_n}
(z-L_{bc})^{-1} dz, \quad C_n=\{|z-n^2|= n\},
\end{equation}
with the corresponding Riesz projections $P_n^0$ of the free
operator $L_{bc}^0$  (although $E_n^0 = Ran (P_n^0)$ maybe have no
common nonzero vectors with the domain of $L_{bc}).$

The main result is Theorem \ref{thm1}, which claims that
\begin{equation}
\label{05} \tilde{\tau}_n = \|P_n -P_n^0 \|_{L^1 \to L^\infty} \to
0.
\end{equation}

This implies  a sort of {quantum chaos}, namely all $L^p$--norms on
the Riesz subspaces $E_n = Ran P_n, $ for bc = $Per^\pm $ or $Dir, $
are uniformly equivalent (see Theorem 6 in Section 5).

In our analysis (see \cite{DM15}) of the relationship between
smoothness of a potential $v $ and the rate of decay of spectral
gaps and spectral triangles  a statement similar to (\ref{05})
\begin{equation}
\label{06} \tau_n = \|P_n -P_n^0 \|_{L^2 \to L^\infty} \to 0.
\end{equation}
was crucial when we used the deviations of Dirichlet eigenvalues
from periodic or anti--periodic eigenvalues to estimate the Fourier
coefficients of the potentials $v.$ But if $v \in L^2 $ it was
''easy'' (see \cite{DM5}, Section 3, Prop.4, or \cite{DM15},
Prop.11). Moreover, those are strong estimates: for $n \geq
N(\|v\|_{L^2})$
\begin{equation}
\label{07} \tau_n \leq \frac{C}{n} \|v\|_{L^2},
\end{equation}
where $C$ is an absolute constant. Therefore, in (\ref{07}) only the
$L^2$--norm is important, so $\tau_n \leq CR/n$ holds {\em for every
$v$ in an $L^2$--ball of radius } $R.$

Just for comparison let us mention the same type of question in the
case of 1D periodic Dirac operators
$$
MF = i \begin{pmatrix} 1 &0\\0  & -1
\end{pmatrix}
\frac{dF}{dx} +  \begin{pmatrix} 0 &p\\ q  & 0
\end{pmatrix}   F,  \quad    0 \leq x\leq \pi,
$$
where $p$ and $q$ are $L^2$--functions and $ F= \begin{pmatrix}
f_1\\f_2
\end{pmatrix}.$
The boundary conditions under consideration are $Per^\pm $ and
$Dir,$ where
$$Per^\pm : \;\; F(\pi) = \pm F(0), \qquad Dir:  \; \; f_1 (0)=f_2
(0), \;\; f_1 (\pi) = f_2 (\pi). $$ Then (see \cite{M04} or
\cite{DM15}, Section 1.1)
$$
E_n^0 = \left \{ \begin{pmatrix} a e^{-inx} \\be^{inx}
\end{pmatrix} :\;\; a,b \in \mathbb{C} \right \},\quad n \in
\mathbb{Z},
$$
where $n$ is even if $bc = Per^+$ and $n$ is odd if $bc = Per^-,$
and $$ E_n^0 = \{c \sin nx, \; c \in \mathbb{C} \}, \quad n \in
\mathbb{N} $$ if $ bc =Dir.$ Then for $$ Q_n = \frac{1}{2\pi i}
\int_{C_n} (\lambda - L)^{-1} d\lambda, \quad C_n = \{\lambda: \;
|\lambda - n| = 1/4 \},$$ we have
$$
\rho_n (V) := \| Q_n -Q_n^0\|_{L^2 \to L^\infty } \to 0;
$$
moreover, for any compact set $K \subset L^2$ and $ V\in K,$ i.e.,
$p,q \in K $ one can construct a sequence $\varepsilon_n (K) \to 0$
such that $\rho_n (V) \leq \varepsilon_n (K), \; V \in K. $ This has
been proven in \cite{M04}, Prop.8.1 and Cor.8.6;  see Prop. 19 in
\cite{DM15} as well.

Of course, the norms $\tau_n$ in (\ref{06}) are larger than the
norms of these operators in $L^2$
$$
t_n = \| P_n -P_n^0\|_{L^2 \to L^2 } \leq \tau_n
$$
and better (smaller) estimates for $t_n $ are possible. For example,
A. Savchuk and A. Shkalikov proved (\cite{SS03}, Sect.2.4) that
$\sum t^2_n < \infty. $ This implies (by Bari--Markus theorem -- see
\cite{GK}, Ch.6, Sect.5.3, Theorem 5.2) that the spectral
decompositions
$$
f = f_N + \sum_{n>N} P_n f
$$
converge unconditionally. For Dirac operators the Bari--Markus
condition is
$$
\sum_{n\in \mathbb{Z},|n|>N} \|Q_n - Q_n^0\|^2 < \infty. $$ This
fact (and completeness of the system of Riesz subspaces $Ran \,Q_n$)
imply unconditional convergence of the spectral decompositions. This
has been proved in \cite{M04} under the assumption that the
potential $V$ is in the Sobolev space $H^\alpha, \; \alpha>1/2 $
(see \cite{M04}, Thm 8.8 for more precise statement). See further
comments in Section 5 below as well.

The proof of Theorem \ref{thm1}, or the estimates of norms
(\ref{05}), are based on the perturbation theory, which gives the
representation
\begin{equation}
\label{011}
 P_n -P_n^0=
\frac{1}{2\pi i} \int_{C_n} \left ( R(\lambda) -R^0 (\lambda) \right
) d\lambda,
\end{equation}
where $R(\lambda) = (\lambda - L_{bc})^{-1} $ and $R^0 (\lambda )$
are the resolvents of $L_{bc}$ and of the free operator $L^0_{bc},$
respectively. Often -- and certainly in the above mentioned examples
where $v \in L^2 $   --  one can get reasonable estimates for the
norms $ \| R(\lambda ) - R^0 (\lambda )\|$ on the contour $C_n,$ and
then by integration for $\|P_n -P_n^0 \|.$ But now, with  $v \in
H^{-1}, $ we succeed to get good estimates for the norms $\|P_n
-P_n^0 \|$ {\em after} having integrated term by term the series
representation
\begin{equation}
\label{012} R-R^0 = R^0VR^0  +R^0VR^0 V R^0  + \cdots .
 \end{equation}
This integration kills or makes more manageable many terms, maybe in
their matrix representation. Only then we go to the norm estimates.
Technical details of this procedure (Section 3) is the core of the
proof of Theorem \ref{thm1}, and of this paper. \vspace{3mm}

{\em Acknowledgements.} This paper has been completed in Fall
Semester 2007 when Boris Mityagin stayed at Weizmann Institute of
Science, Rehovot, Israel, as Weston Visiting Professor; he thanks
Weizmann Institute for hospitality and stimulating environment.

\section{Main result}

 By our Theorem 21 in \cite{DM16} (about spectra localization), the operator
$L_{Per\pm}$ has, for large enough $n,$ exactly two eigenvalues
(counted with their algebraic multiplicity) inside the disc of
radius $n$ about $n^2$ (periodic for even $n$ or antiperiodic for
odd $n$). The operator $L_{Dir}$ has one eigenvalue in these discs
for all large enough $n.$

Let $E_n$  be the corresponding Riesz invariant subspace, and let
$P_n$  be the corresponding Riesz projection, i.e., $$ P_n =
\frac{1}{2\pi i} \int_{C_n} (\lambda - L)^{-1} d\lambda, $$ where
$C_n = \{\lambda: \; |\lambda - n^2| =n \}.   $ We denote by $P_n^0$
the Riesz projector that corresponds to the free operator.

\begin{Proposition}
\label{prop1} In the above notations, for boundary conditions $bc
=Per^\pm $ or $Dir,$
\begin{equation}
\label{p20} \|P_n - P_n^0 \|_{L^2 \to L^\infty} \to 0 \quad
\text{as} \;\; n \to \infty.
\end{equation}
\end{Proposition}

As a matter of fact we will prove a stronger statement.

\begin{Theorem}
\label{thm1} In the above notations, for boundary conditions $bc
=Per^\pm $ or $Dir,$
\begin{equation}
\label{p21} \|P_n - P_n^0 \|_{L^1 \to L^\infty} \to 0 \quad
\text{as} \;\; n \to \infty.
\end{equation}
\end{Theorem}

\begin{proof}
We give a complete proof in the case $bc =Per^\pm.$ If $bc =Dir$ the
proof is the same, and only minor changes are necessary due to the
fact that in this case the orthonormal system of eigenfunctions of
$L^0$ is $\{ \sqrt{2} \sin nx, \; n\in \mathbb{N}\}$ ( while it is
$\{\exp(imx), \; m\in 2 \mathbb{Z}\}$ for $bc = Per^+, $ and
$\{\exp(imx), \; m\in 1+ 2 \mathbb{Z}\}$ for $bc = Per^- $). So,
roughly speaking, the only difference is that when working with $bc
= Per^\pm $ the summation indexes in our formulas below run,
respectively, in $2 \mathbb{Z}$ and $1+ 2 \mathbb{Z},$ while for $bc
=Dir$ the summation indexes have to run in $\mathbb{N}.$ Therefore,
we consider in detail only $bc= Per^\pm, $ and provide some formulas
for the case $bc=Dir.$

Let
\begin{equation}
\label{p210} B_{km}(n):= \langle (P_n - P_n^0) e_m, e_k \rangle.
\end{equation}
We are going to prove that
\begin{equation}
\label{p22} \sum_{k,m} |B_{km}(n)|    \to 0 \quad \text{as} \;\; n
\to \infty.
\end{equation}
Of course, the convergence of the series in (\ref{p22}) means that
the operator with the matrix $B_{km} (n)$ acts from $\ell^\infty $
into $\ell^1.$

The Fourier coefficients of an $L^1$-function form an $\ell^\infty
$-sequence. On the other hand,
\begin{equation}
\label{p220} D= \sup_{x,n} |e_n (x)|  < \infty.
\end{equation}
 Therefore, the operators $P_n -P_n^0 $ act
from $L^1$ into $L^\infty $ (even into $C$) and
\begin{equation}
\label{p23} \|P_n - P_n^0 \|_{L^1 \to L^\infty} \leq D^2 \sum_{k,m}
|B_{km}(n)|.
\end{equation}
Indeed, if  $\|f\|_{L^1} =1$ and  $f = \sum f_m e_m, $ then
$|f_m|\leq D $ and
$$ (P_n - P_n^0 )f = \sum_k \left ( \sum_m B_{km} f_m  \right ) e_k.
$$

Taking into account (\ref{p220}), we get
$$
\|(P_n - P_n^0) f \|_{L^\infty} \leq D\sum_k \left | \sum_m B_{km}
f_m \right | \leq  D^2\sum_k \sum_m |B_{km}|,
$$
which proves (\ref{p23}).

In \cite{DM16}, Section 5, we gave a detailed analysis of the
representation
$$ R_\lambda - R_\lambda^0 = \sum_{s=0}^\infty K_\lambda
(K_\lambda V  K_\lambda )^{s+1} K_\lambda, $$ where $K_\lambda =
\sqrt{R^0_\lambda} $ -- see \cite{DM16}, (5.13-14) and what follows
there. By (\ref{011}),
$$ P_n - P_n^0 = \frac{1}{2\pi i} \int_{C_n} \sum_{s=0}^\infty
K_\lambda   (K_\lambda V  K_\lambda )^{s+1} K_\lambda d\lambda. $$
if the series on the right converges. Thus
\begin{equation}
\label{p24} \langle (P_n - P_n^0) e_m, e_k \rangle =
   \sum_{s=0}^\infty
\frac{1}{2\pi i} \int_{C_n} \langle K_\lambda  (K_\lambda V
K_\lambda )^{s+1} K_\lambda e_m, e_k \rangle d\lambda,
\end{equation}
so we have
\begin{equation}
\label{p24b} \sum_{k,m} |\langle (P_n - P_n^0) e_m, e_k \rangle|
\leq \sum_{s=0}^\infty A(n,s),
\end{equation}
where
\begin{equation}
\label{p24c} A(n,s) =\sum_{k,m} \left | \frac{1}{2\pi i}
\int_{C_n} \langle K_\lambda (K_\lambda V K_\lambda )^{s+1}
K_\lambda e_m, e_k \rangle d\lambda \right |.
\end{equation}

By the matrix representation of the operators $K_\lambda $ and $V$
(see more details in \cite{DM16}, (5.15-22)) it follows that
\begin{equation}
\label{p24a} \langle K_\lambda (K_\lambda V  K_\lambda )K_\lambda
e_m, e_k \rangle = \frac{V(k-m)}{(\lambda -k^2)(\lambda -m^2)},
\quad k,m  \in n+2\mathbb{Z},
\end{equation}
for $bc = Per^\pm, $ and
\begin{equation}
\label{p25a} \langle K_\lambda (K_\lambda V  K_\lambda )K_\lambda
e_m, e_k \rangle =
\frac{|k-m|\tilde{q}(|k-m|)-(k+m)\tilde{q}(k+m)}{\sqrt{2}(\lambda
-k^2)(\lambda -m^2)}, \quad k,m \in \mathbb{N},
\end{equation}
for $bc = Dir. $ Let us remind that $\tilde{q}(m)$ are the sine
Fourier coefficients of the function $Q(x),$ i.e.,
$$ Q(x) = \sum_{m=1}^\infty \tilde{q}(m) \sqrt{2} \sin mx.
$$
The matrix representations of $K_\lambda (K_\lambda V  K_\lambda
)K_\lambda$ in (\ref{p24a}) and (\ref{p24b}) are the ''building
blocks'' for the matrices of the products of the form $K_\lambda
(K_\lambda V  K_\lambda )^s K_\lambda$ that we have to estimate
below. For convenience, we set
\begin{equation}
\label{p26} V(m) = m w(m), \quad w \in \ell^2 (2\mathbb{Z}), \quad
r(m)= max(|w(m)|,|w(-m)|)
\end{equation}
if $bc = Per^\pm, $ and
\begin{equation}
\label{p260} \tilde{q}(0)=0, \quad r(m) = \tilde{q}(|m|), \quad m
\in \mathbb{Z}.
\end{equation}
if $bc = Dir.$ We use the notations (\ref{p26}) in the estimates
related to $bc =Per^\pm $ below, and if one would use in a similar
way (\ref{p260}) in the Dirichlet case, then the corresponding
computations becomes practically identical (the only difference will
be that in the Dirichlet case the summation will run over
$\mathbb{Z}$). So, further we consider only  the case $bc =Per^\pm.
$

Let us calculate the first term on the right--hand side of
(\ref{p24}) (i.e., the term coming for $s=0$). We have
\begin{equation}
\label{p25} \frac{1}{2\pi i} \int_{C_n} \frac{V(k-m)}{(\lambda
-k^2)(\lambda -m^2)}d\lambda = \begin{cases} \frac{V(k \mp n)}{(n^2
-k^2)}   &  m= \pm n, \;\; k \neq \pm n, \\ \frac{V(\pm n -m)}{(n^2
-m^2)}   &  k= \pm n, \;\; m \neq \pm n, \\ 0 &   \text{otherwise}.
\end{cases}
\end{equation}
Thus $$ A(n,0) =\sum_{k,m} \left | \frac{1}{2\pi i} \int_{C_n}
\langle K_\lambda (K_\lambda V  K_\lambda )K_\lambda e_m, e_k
\rangle \right | $$ $$ = \sum_{k \neq \pm n} \frac{|V(k-n)|}{|n^2
-k^2|} + \sum_{k \neq \pm n} \frac{|V(k+n)|}{|n^2 -k^2|} + \sum_{m
\neq \pm n} \frac{|V(-n+m)|}{|n^2 -m^2|} + \sum_{m \neq \pm n}
\frac{|V(n-m)|}{|n^2 -m^2|}. $$

By the Cauchy inequality, we estimate the first sum on the
right--hand side:
\begin{equation}
\label{p26a} \sum_{k \neq \pm n} \frac{|V(k-n)|}{|n^2 -k^2|}
=\sum_{k \neq \pm n} \frac{ |k-n| |w(k-n)|}{|n^2 -k^2|}
\end{equation}
$$
\leq \sum_{k \neq -n} \frac{r(k-n)}{|n+k|}  \leq \sum_{k>0} \cdots +
\sum_{k\leq 0, k \neq  -n} \cdots $$

$$ \leq \left ( \sum_{k>0} \frac{1}{|n+k|^2} \right )^{1/2} \cdot
\|r\|+ \left ( \sum_{k\leq 0, k \neq  -n} \frac{1}{|n+k|^2} \right
)^{1/2}  \left ( \sum_{k\leq 0} (r(n-k))^2 \right )^{1/2} $$ $$ \leq
\frac{\|r\|}{\sqrt{n}} + \mathcal{E}_n (r). $$ Since each of the
other three sums could be estimated in the same way, we get
\begin{equation}
\label{p27} A(n,0) \leq \sum_{k,m}  \left | \frac{1}{2\pi i}
\int_{C_n} \langle K_\lambda (K_\lambda V  K_\lambda )K_\lambda e_m,
e_k \rangle   d \lambda \right |  \leq \frac{4\|r\|}{\sqrt{n}} + 4
\mathcal{E}_n (r).
\end{equation}

Next we estimate $A(n,s), s\geq 1.$ By the matrix representation of
$K_\lambda $ and $V$ -- see (\ref{p24a}) -- we have
\begin{equation}
\label{p27a}
\langle K_\lambda  (K_\lambda V  K_\lambda )^{s+1} K_\lambda e_m, e_k
\rangle
= \frac{\Sigma (\lambda;s,k,m)}{(\lambda - k^2)(\lambda - m^2)}
\end{equation}
where
\begin{equation}
\label{p28}
\Sigma (\lambda;s,k,m)
=\sum_{j_1, \ldots, j_s}
\frac{V(k-j_1) V(j_1 - j_2) \cdots V(j_{s-1} -j_s)V(j_s -m)}
{(\lambda -j_1^2) (\lambda -j_2^2)\cdots
 (\lambda -j_s^2) },
\end{equation}
$ k, m, j_1, \ldots, j_s \in n+ 2\mathbb{Z}. $
For convenience, we set also
\begin{equation}
\label{p28a}
\Sigma (\lambda;0,k,m) = V(k-m).
\end{equation}

In view of (\ref{p24c}), we have
\begin{equation}
\label{p29} A(n,s) = \sum_{k,m} \left |
 \frac{1}{2\pi i} \int_{C_n}
\frac{\Sigma (\lambda;s,k,m)}{(\lambda - k^2)(\lambda - m^2)}
d\lambda \right |.
\end{equation}

Let us consider the following sub--sums of $\Sigma (\lambda;s,k,m):$
\begin{equation}
\label{p30}
\Sigma^0 (\lambda;s,k,m) =\sum_{j_1, \ldots, j_s \neq \pm n} \cdots \quad
\text{for} \;s \geq 1,
\quad \Sigma^0 (\lambda;0,k,m) := V(k-m);
\end{equation}
\begin{equation}
\label{p30a}
\Sigma^1 (\lambda;s,k,m)
= \sum_{\exists \; \text{one} \; j_\nu = \pm n}  \cdots
\quad  \text{for} \;s \geq 1;
\end{equation}
\begin{equation}
\label{p31}
\Sigma^* (\lambda;s,k,m) = \sum_{\exists j_\nu = \pm n}  \cdots, \quad
\Sigma^{**} (\lambda;s,k,m)
= \sum_{\exists j_\nu, j_\mu = \pm n}  \cdots, \quad s \geq 2
\end{equation}
(i.e.,  $\Sigma^0 $ is the sub--sum of $\Sigma $ over those indices
$j_1, \ldots, j_s$ that are different from $\pm n ,$  in $\Sigma^1 $
exactly one summation index is equal to $\pm n,$ in $\Sigma^* $ at
least one summation index is equal to $\pm n,$ and in $\Sigma^{**} $
at least two summation indices are equal to $\pm n).$  Notice that
$$ \Sigma (\lambda;s,k,m)=\Sigma^0 (\lambda;s,k,m) + \Sigma^*
(\lambda;s,k,m), \quad s\geq 1, $$ and
$$  \Sigma (\lambda;s,k,m) = \Sigma^0 (\lambda;s,k,m) +\Sigma^1
(\lambda;s,k,m) + \Sigma^{**} (\lambda;s,k,m), \quad s\geq 2. $$

In these notations we have
\begin{equation}
\label{p32}
\sum_{m,k \neq \pm n} \left |
 \frac{1}{2\pi i} \int_{C_n}
\frac{\Sigma^0 (\lambda;s,k,m)}{(\lambda - k^2)(\lambda - m^2)}
d\lambda \right | = 0
\end{equation}
because, for $m,k \neq \pm n,$ the integrand is an analytic function
of $\lambda $ in the disc $\{\lambda: \; |\lambda - n^2| \leq n/4
\}.$

Therefore, $A(n,s)$ could be estimated as follows:
\begin{equation}
\label{p33a} A(n,1) \leq  \sum_{i=1}^5 A_i (n,1),
\end{equation}
and
\begin{equation}
\label{p33} A(n,s) \leq  \sum_{i=1}^7 A_i (n,s), \quad s \geq 2,
\end{equation}
where
\begin{equation}
\label{p331} A_1 (n,s) = \sum_{k,m = \pm n} n \cdot \sup_{\lambda
\in C_n } \left | \frac{\Sigma (\lambda;s,k,m)}{(\lambda -
k^2)(\lambda - m^2)} \right |,
\end{equation}
\begin{equation} \label{p332}
A_2
(n,s) = \sum_{k= \pm n, m \neq \pm n} n \cdot \sup_{\lambda \in C_n
} \left | \frac{\Sigma^0 (\lambda;s,k,m)}{(\lambda - k^2)(\lambda -
m^2)} \right |,
\end{equation}
\begin{equation} \label{p333}
A_3 (n,s) = \sum_{k= \pm n, m \neq \pm n} n \cdot \sup_{\lambda \in
C_n } \left | \frac{\Sigma^* (\lambda;s,k,m)}{(\lambda -
k^2)(\lambda - m^2)} \right |,
\end{equation}
\begin{equation} \label{p334}
A_4 (n,s) = \sum_{k \neq \pm n, m = \pm n} n \cdot \sup_{\lambda \in
C_n } \left | \frac{\Sigma^0 (\lambda;s,k,m)}{(\lambda -
k^2)(\lambda - m^2)} \right |,
\end{equation}
\begin{equation} \label{p335}
A_5 (n,s) = \sum_{k \neq \pm n, m = \pm n} n \cdot \sup_{\lambda \in
C_n } \left | \frac{\Sigma^* (\lambda;s,k,m)}{(\lambda -
k^2)(\lambda - m^2)} \right |,
\end{equation}
\begin{equation} \label{p336}
A_6 (n,s) = \sum_{k,m \neq \pm n} n \cdot \sup_{\lambda \in C_n }
\left | \frac{\Sigma^1 (\lambda;s,k,m)}{(\lambda - k^2)(\lambda -
m^2)}\right |,
\end{equation}
\begin{equation} \label{p337}
 A_7 (n,s) = \sum_{k,m \neq \pm n} n \cdot
\sup_{\lambda \in C_n } \left | \frac{\Sigma^{**}
(\lambda;s,k,m)}{(\lambda - k^2)(\lambda - m^2)} \right |.
\end{equation}
First we estimate $A_1 (n,s). $ By (\ref{p24a}) and \cite{DM16},
Lemma 19 (inequalities (5.30),(5.31)),
\begin{equation}
\label{p34} \sup_{\lambda \in C_n} \|K_\lambda \|
=\frac{2}{\sqrt{n}}, \quad \sup_{\lambda \in C_n} \| K_\lambda V
K_\lambda \| \leq \rho_n := C \left (\frac{\|r\|}{\sqrt{n}} +
\mathcal{E}_{\sqrt{n}} (r)\right  ),
\end{equation}
where $r = (r(m))$ is defined by the relations (\ref{p26}) and $C$
is an absolute constant.

\begin{Lemma}
\label{lemp0} In the above notations
\begin{equation}
\label{p36} \sup_{\lambda \in C_n } \left | \frac{\Sigma
(\lambda;s,k,m)}{(\lambda - k^2)(\lambda - m^2)} \right | \leq
\frac{1}{n} \rho_n^{s+1}.
\end{equation}
\end{Lemma}

\begin{proof}
Indeed, in view of (\ref{p28}) and  (\ref{p34}), we have
$$ \left |
\frac{\Sigma (\lambda;s,k,m)}{(\lambda - k^2)(\lambda - m^2)} \right |
=
|\ K_\lambda  (K_\lambda V  K_\lambda )^{s+1} K_\lambda e_k,
e_m \rangle | $$
$$
\leq \| K_\lambda  (K_\lambda V  K_\lambda )^{s+1} K_\lambda \|
\leq  \| K_\lambda \| \cdot
 \| K_\lambda V  K_\lambda \|^{s+1} \cdot  \| K_\lambda \| \leq
\frac{1}{n} \rho_n^{s+1},
$$
which proves (\ref{p36}).
\end{proof}

Now we estimate $A_1 (n,s).$ By (\ref{p36}),
\begin{equation}
\label{p37} A_1 (n,s) = \sum_{m,k = \pm n} n \cdot \sup_{\lambda
\in C_n } \left | \frac{\Sigma (\lambda;s,k,m)}{(\lambda -
k^2)(\lambda - m^2)} \right | \leq  4 \rho_n^{s+1}.
\end{equation}

To estimate $A_2 (n,s),$ we consider $\Sigma^0 (\lambda; s,k,m)$
for $k = \pm n.$  From the elementary inequality
\begin{equation}
\label{p38} \frac{1}{|\lambda - j^2|} \leq \frac{2}{|n^2 - j^2|}
\quad \text{for} \quad  \lambda \in C_n, \; j \in n + 2
\mathbb{Z}, \; j \neq \pm n,
\end{equation}
it follows, for  $m \neq \pm n,$
\begin{equation}
\label{p39}
\sup_{\lambda \in C_n } \left |
\frac{\Sigma^0 (\lambda;s,\pm n,m)}{(\lambda - n^2)(\lambda - m^2)} \right |
\leq
\frac{1}{n} \cdot 2^{s+1} \times
\end{equation}
$$
\times
  \sum_{j_1,\ldots, j_s \neq \pm n} \frac{|V(\pm n -j_1 )V(j_1
-j_2) \cdots V(j_{s-1}- j_s)V(j_s -m)|}
{|n^2 - j_1^2 | |n^2- j_2^2 |\cdots
|n^2- j_s^2||n^2 -m^2 |}.
$$
Thus, taking the sum of both sides of (\ref{p39}) over $m \neq \pm n, $
we get

\begin{equation}
\label{p40} A_2 (n,s) \leq 2^{s+1} \left [ L(s+1,n) + L(s+1,-n)
\right ],
\end{equation}
where
\begin{equation}
\label{p41}
L(p,d) :=
  \sum_{i_1,\ldots, i_p \neq \pm n} \frac{|V(d-i_1 )|}{|n^2 - i_1^2 |}
\cdot \frac{|V(i_1-i_2 )|}{|n^2 - i_2^2 |} \cdots
\frac{|V(i_{p-1}-i_p )|}{|n^2 - i_p^2 |}.
\end{equation}

The roles of $k$ and $m$ in $A_2 (n,s) $ and $A_4 (n,s) $ are
symmetric, so $A_4 (n,s) $ could be estimated in an analogous way.
Indeed, for $k \neq \pm n,$ we have
\begin{equation}
\label{p39a}
\sup_{\lambda \in C_n } \left |
\frac{\Sigma^0 (\lambda;s,k,\pm n)}{(\lambda - k^2)(\lambda - n^2)} \right |
\leq
\frac{1}{n} \cdot 2^{s+1} \times
\end{equation}
$$
\times
  \sum_{j_1,\ldots, j_s \neq \pm n} \frac{|V(k -j_1 )V(j_1
-j_2) \cdots V(j_{s-1}- j_s)V(j_s -\pm n)|}
{|n^2 - k^2 ||n^2 - j_1^2 | |n^2- j_2^2 |\cdots
|n^2- j_s^2|}.
$$
Thus, taking the sum of both sides of (\ref{p39a}) over $k \neq \pm n, $
we get
\begin{equation}
\label{p42} A_4 (n,s) \leq 2^{s+1} \left [ R(s+1,n) + R(s+1,-n)
\right ],
\end{equation}
where
\begin{equation}
\label{p43}
R(p,d) :=
  \sum_{i_1,\ldots, i_p \neq \pm n}
 \frac{|V(i_1-i_2 )|}{|n^2 - i_1^2 |} \cdots
\frac{|V(i_{p-1}-i_p )|}{|n^2 - i_{p-1}^2 |} \cdot
\frac{|V(i_p - d )|}{|n^2 - i_p^2 |}.
\end{equation}

Below (see Lemma \ref{lemp1} and its proof in Sect. 3)  we estimate
the sums $ L(p,\pm n) $ and $ R(p,\pm n). $ But now we are going to
show that $A_i (n,s), \; i = 3, 5, 6, 7, $ could be estimated in
terms of $L$ and $R$ from (\ref{p41}), (\ref{p43}) as well.

To estimate $ A_6 (n,s)$ we write the expression $\frac{\Sigma^1
(\lambda;s,k,m)} {(\lambda - k^2)(\lambda - m^2)} $ in the form $$
\sum_{\nu =1}^s \sum_{d =\pm n} \frac{1}{\lambda- k^2}\Sigma^0
(\lambda; \nu -1, k, d) \frac{1}{\lambda- n^2}\Sigma^0 (\lambda;s -
\nu, d, m) \frac{1}{\lambda- m^2} $$ By (\ref{p38}), the absolute
values of the terms of this double sum do not exceed:

(a) for $\nu = 1 $
$$
2^{s+1} \cdot \frac{|V(k- \pm n)|}{|n^2 - k^2|} \cdot \frac{1}{n} \cdot
\sum_{i_1, \ldots,  i_{s-1}\neq \pm n}
\frac{|V(\pm n - i_1)|  |V(i_1 -i_2) | \cdots  |V(i_{s-1} -m) |}
{|n^2 -i_1^2| \cdots  |n^2 -i_{s-1}^2|  |n^2 -m^2|}.
$$

(b) for $\nu = s $
$$
2^{s+1} \cdot   \left (
\sum_{i_1, \ldots,  i_{s-1}\neq \pm n}
\frac{|V(k - i_1)|
|V(i_1 - i_2)| \cdots  |V(i_{s-1} -\pm n)|}
{|n^2 -k^2||n^2 -i_1^2| |n^2 -i_2^2|
\cdots  |n^2 -i_{s-1}^2| } \right )
\cdot \frac{1}{n} \cdot
\frac{|V(\pm n -m)|}{|n^2 - m^2|}
$$

(c) for $ 1 < \nu < s $ $$ 2^{s+1} \cdot   \left ( \sum_{i_1,
\ldots,  i_{\nu-1}\neq \pm n} \frac{|V(k - i_1)| |V(i_1 - i_2)|
\cdots  |V(i_{\nu -1} -\pm n)|} {|n^2 -k^2| |n^2 -i_1^2| |n^2
-i_2^2| \cdots  |n^2 -i_{\nu -1}^2| } \right ) \cdot \frac{1}{n}
$$ $$    \times \; \sum_{i_1, \ldots,  i_{s-\nu} \neq \pm n}
\frac{|V(\pm n - i_1)|  |V(i_1 -i_2) | \cdots  |V(i_{s-\nu} -m) |}
{|n^2 -i_1^2| \cdots  |n^2 -i_{s-\nu}^2|  |n^2 -m^2|}. $$

Therefore, taking the sum over $m,k \neq \pm n, $
we get
\begin{equation}
\label{p44} A_6 (n,s) \leq 2^{s+1} \cdot \sum_{\nu =1}^s \sum_{d =
\pm n} R(\nu,d) \cdot L(s+1-\nu ,d).
\end{equation}

One could estimate  $A_3(n,s), A_5 (n,s) $ and $A_7 (n,s) $ in an
analogous way. We will write the core formulas but omit some
details.

To estimate $A_3 (n,s), $ we use the identity $$ \frac{\Sigma
(\lambda;s,k, \pm n )}{(\lambda -k^2)(\lambda - n^2)} =\sum_{\nu
=1}^s \sum_{d=\pm n} \frac{1}{\lambda - k^2} \Sigma^0 (\lambda;\nu
-1, k, d) \frac{1}{\lambda - n^2} \Sigma (\lambda;s- \nu, d, \pm
n) \frac{1}{\lambda - n^2}. $$ In view of (\ref{p36}), (\ref{p38})
and (\ref{p43}), from here it follows that
\begin{equation}
\label{p45} A_3 (n,s) \leq 2^{s+1} \cdot \sum_{\nu =1}^s \sum_{d =
\pm n}
 R(\nu ,d) \cdot \rho_n^{s-\nu+1}.
\end{equation}

We estimate $A_5 (n,s) $ by using the identity $$ \frac{\Sigma
(\lambda;s,\pm n, m)}{(\lambda -n^2)(\lambda - m^2)} =\sum_{\nu
=1}^s \sum_{d=\pm n} \frac{1}{\lambda - n^2} \Sigma (\lambda;\nu -1,
\pm n, d) \frac{1}{\lambda - n^2} \Sigma^0 (\lambda;s- \nu, d, m)
\frac{1}{\lambda - m^2}. $$ In view of (\ref{p36}), (\ref{p38}) and
(\ref{p41}), from here it follows that
\begin{equation}
\label{p46} A_5 (n,s) \leq 2^{s+1} \cdot \sum_{\nu =1}^s \sum_{d =
\pm n}
 \rho_n^{\nu} \cdot L(s-\nu+1,d).
\end{equation}

Finally, to estimate $A_7 (n,s) $ we use the identity
$$\frac{\Sigma (\lambda;s,k, m)}{(\lambda -k^2)(\lambda - m^2)} =
\sum_{1 \leq \nu < \mu \leq s }^s \sum_{d_1, d_2=\pm n}
\frac{1}{\lambda - k^2} \Sigma^0 (\lambda;\nu -1, k, d_1)  \times
$$$$ \times \; \frac{1}{\lambda - n^2}  \Sigma (\lambda; \mu- \nu
-1, d_1, d_2) \frac{1}{\lambda - n^2} \Sigma^0 (\lambda;s -\mu,
d_2, m) \frac{1}{\lambda - m^2} $$ In view of (\ref{p36}),
(\ref{p38}), (\ref{p41}) and (\ref{p43}), from here it follows
that
\begin{equation}
\label{p47} A_7 (n,s) \leq 2^{s} \cdot \sum_{1 \leq \nu < \mu \leq s
} \sum_{d_1, d_2=\pm n}  R(\nu,d_1) \cdot \rho_n^{\mu-\nu} \cdot
L(s-\mu+1,d_2).
\end{equation}

Next we estimate $L(p,\pm n) $ and $R(p,\pm n).$ Changing the
indices in (\ref{p43}) by $$ j_\nu = -i_{p+1-\nu}, \quad  1\leq
\nu \leq p, $$ we get
\begin{equation}
\label{p51} R(p,d) = L(p, -d).
\end{equation}
Therefore, it is enough to estimate only $L(p, \pm n).$

\begin{Lemma}
\label{lemp1} In the above notations, there exists a sequence of
positive numbers $\varepsilon_n \to 0 $ such that, for large enough
$n,$
\begin{equation}
\label{p52}  L(s, \pm n) \leq (\varepsilon_n)^s, \quad \forall
\,s\in \mathbb{N}.
\end{equation}

\end{Lemma}

The proof of this lemma is technical. It is given in detail in
Section 3. Then in Section 4 we complete the proof of Theorem
\ref{thm1}. With (\ref{p51}) and (\ref{p52}), in Section 4 we will
use Lemma \ref{lemp1} in the following form.

\begin{Corollary}
\label{cor1}
 In the above notations, there exists a sequence of
positive numbers $\varepsilon_n \to 0 $ such that, for large enough
$n,$
\begin{equation}
\label{p520} \max\{ L(s, \pm n), R(s, \pm n) \} \leq
(\varepsilon_n)^s, \quad \forall \,s\in \mathbb{N}.
\end{equation}
\end{Corollary}

\section{Proofs and technical inequalities}

We follow the notations from Section 2. Now we prove Lemma
\ref{lemp1}.

\begin{proof}
First we show that
\begin{equation}
\label{p53}  L(s, \pm n) \leq \sigma (n,s),\quad  s \geq 1,
\end{equation}
where
\begin{equation}
\label{p11a}  \sigma (n,1) =\sum_{j_1 \neq \pm n}
\frac{r(n+j_1)}{|n^2 - j_1^2|},
\end{equation}
for $s\geq 2 $
\begin{equation}
\label{p11} \sigma (n,s) := \sum_{j_1,\ldots, j_s \neq \pm n}
\left (\frac{1}{|n-j_1|} +\frac{1}{|n+j_2|} \right )  \cdots \left
(\frac{1}{|n-j_{s-1}|} +\frac{1}{|n+j_s|} \right )
\frac{1}{|n-j_s|}
\end{equation}
$$ \times \;\;  r(n+j_1)r(j_1 +j_2) \cdots r(j_{s-1}+j_s), $$ and
the sequence $r= (r(m)) $ is defined by (\ref{p26}).

For $s=1$ we have, with $i_1 = -j_1,$
$$
L(1,n) = \sum_{j_1 \neq \pm n} \frac{|V(n-j_1)|}{|n^2 - j_1^2|}
=\sum_{i_1 \neq \pm n} \frac{|V(n+i_1)|}{|n^2 - i_1^2|}
=\sum_{i_1 \neq \pm n} \frac{|w(n+i_1)|}{|n - i_1|}
\leq \sum_{i_1 \neq \pm n} \frac{r(n+i_1)}{|n - i_1|}
$$
(where (\ref{p26}) is used).
In an analogous way we get
$$
L(1,-n) = \sum_{j_1 \neq \pm n} \frac{|V(-n-j_1)|}{|n^2 - j_1^2|}
=\sum_{i_1 \neq \pm n} \frac{|w(-n-i_1)|}{|n - i_1|}
\leq \sum_{i_1 \neq \pm n} \frac{r(n+i_1)}{|n - i_1|},
$$
so, (\ref{p53}) holds for $s=1.$

Let $s \geq 2.$ Changing the indices of summation in (\ref{p41})
(considered with $p=s $ and $d=n)$ by $ j_\nu = (-1)^\nu i_\nu, $
we get $$  L(s,n) = \sum_{j_1,\ldots, j_s \neq \pm n}
\frac{|V(n+j_1)|}{|n^2 -j_1^2 |} \frac{|V(-j_1 -j_2)|}{|n^2 -j_1^2
|} \cdots \frac{|V[ (-1)^{s-1}(j_{s-1}+j_s)]|}{|j_s^2 -n^2|} $$ $$
=\sum_{j_1,\ldots, j_s \neq \pm n} \frac{|n+ j_1||j_2 + j_1|
\cdots |j_s +j_{s-1}|}{|j_1^2 - n^2| |j_2^2 - n^2|\cdots |j_s^2
-n^2|} |w(n+ j_1)w(-j_1 -j_2) \cdots w[ (-1)^{s-1}(j_{s-1}+j_s)]|
$$ $$ \leq
 \sum_{j_1,\ldots, j_s \neq \pm n} \frac{|j_2
+j_1| \cdots |j_s +j_{s-1}|}{|n -j_1| |n^2 -j_2^2 |\cdots |n^2 -
j_s^2 |} r(n+j_1)r(j_1 +j_2) \cdots r(j_{s-1}+j_s). $$

By the identity $$ \frac{i+k}{(n-i)(n+k)} = \frac{1}{n-i}
-\frac{1}{n+k}, $$ we get that the latter sum does not exceed

$$ \sum_{j_1,\ldots, j_s \neq \pm n} \left |\frac{1}{n-j_1}
-\frac{1}{n+j_2} \right |  \cdots \left |\frac{1}{n-j_{s-1}}
-\frac{1}{n+j_s} \right | \frac{1}{|n-j_s|}
$$
$$
 \times  \;r(n+j_1)r(j_1 +j_2) \cdots r(j_{s-1}+j_s)
 \leq \sigma (n,s). $$

Changing the indices of summation in (\ref{p41}) (considered with
$p=s $ and $d=-n)$ by $ j_\nu = (-1)^{\nu +1} i_\nu, $ one can
show that $L(s,-n) \leq \sigma (n,s). $ Since the proof is the
same we omit the details. This completes the proof of (\ref{p53}).

In view of (\ref{p53}), Lemma \ref{lemp1} will be proved if we
show that there exists a sequence of positive numbers
$\varepsilon_n \to 0$ such that, for large enough $n,$
\begin{equation}
\label{p55} \sigma (n,s) \leq  (\varepsilon_n)^s, \quad \forall
\,s\in \mathbb{N}.
\end{equation}

In order to prove (\ref{p55}) we need the following statements.

\begin{Lemma}
\label{lemp2} Let $r = (r(k)) \in \ell^2 (2\mathbb{Z}), \; r(k)
\geq 0, $ and let
\begin{equation}
\label{p0} \sigma_1 (n,s;m) = \sum_{j_1, \ldots, j_s \neq   n}
\frac{r(m+j_1)}{|n-j_1|}\frac{r(j_1+j_2)}{|n-j_2|} \cdots
\frac{r(j_{s-1}+j_s)}{|n-j_s|}, \quad n,s \in \mathbb{N},
\end{equation}
where $ m, j_1, \ldots, j_s \in n+ 2\mathbb{Z}. $ Then, with
\begin{equation}
\label{p1} \tilde{\rho}_n := \mathcal{E}_n (r) + 2\|r\|/\sqrt{n},
\end{equation}
we have, for $ n \geq 4,$
\begin{equation}
\label{p2} \sigma_1 (n,1;m) \leq \begin{cases}
\tilde{\rho}_n  \quad &
\text{if} \;\; |m-n| \leq n/2,\\ \|r\|
\quad & \text{for arbitrary} \; \;
m \in n+2\mathbb{Z},
\end{cases}
\end{equation}
\begin{equation}
\label{p3} \sigma_1 (n,2p;m) \leq  (2\|r\| \tilde{\rho}_n)^p, \quad
\sigma_1 (n,2p + 1;m) \leq \|r\| \cdot (2\|r\| \tilde{\rho}_n)^p.
\end{equation}

\end{Lemma}

\begin{proof}  Let us recall that
$$\sum_{k=1}^\infty \frac{1}{k^2} = \pi^2/6, \quad
\sum_{k=n+1}^\infty \frac{1}{k^2} <\sum_{k=n+1}^\infty \left
(\frac{1}{k-1} - \frac{1}{k} \right ) = \frac{1}{n}. $$ Therefore,
one can easily see that
 $$
\sum_{i\in 2\mathbb{Z}, i\neq 0} \frac{1}{i^2} = \pi^2/12 < 1,
\quad \sum_{i\in 2\mathbb{Z}, |i|> n/2} \frac{1}{i^2} < 4/n, \quad
n \geq 4.$$

By the Cauchy inequality, $$ \sigma_1 (n,1;m) = \sum_{j_1 \neq
n}\frac{r(m+j_1)}{|n-j_1|} \leq  \left ( \sum_{j_1 \neq \pm n}
|n-j_1 |^{-2}  \right  )^{1/2} \cdot \|r\| \leq \|r\|, $$ which
proves the second case in (\ref{p2}).

If $|m-n| \leq  n/2 $ then we have $ n/2 \leq m \leq 3n/2. $ Let
us write $\sigma_1 (n,1;m)$ in the form $$\sigma_1 (n,1;m) =
\sum_{ 0<|j_1-n|\leq n/2} \frac{r(m+j_1)}{|n-j_1|} +
\sum_{|j_1-n|>n/2} \frac{r(m+j_1)}{|n-j_1|} $$ and apply the
Cauchy inequality to each of the above sums. In the first sum  $
n/2 \leq j \leq 3n/2, $ so $ j+m \geq n, $ and therefore, we get
$$ \sigma_1 (n,1;m) \leq \left ( \sum_{ i\geq n} |r(i)|^2 \right
)^{1/2} \cdot 1 + \|r\| \cdot \left ( \sum_{|n-j_1|>n/2} |j_1 -
n|^{-2}   \right )^{1/2}. $$ Thus $$ \sigma_1 (n,1;m) \leq
\mathcal{E}_n (r) + \frac{2\|r\|}{\sqrt{n}} = \tilde{\rho}_n \quad
\text{if} \quad |n-m| \leq n/2. $$ This completes the proof of
(\ref{p2}).

Next we estimate  $\sigma_1 (n,2;m).$ We have $$ \sigma_1 (n,2;m)
= \sum_{j_1 \neq n} \frac{r(m+j_1)}{|n-j_1 |} \sum_{j_2 \neq n}
\frac{r(j_1+j_2)}{|n-j_2|} $$ $$ = \sum_{0<|j_1- n| \leq  n/2}
\frac{r(m+j_1)}{|n-j_1|} \cdot \sigma_1 (n,1;j_1) +
\sum_{|j_1-n|>n/2} \frac{r(m+j_1)}{|n-j_1|} \cdot \sigma_1
(n,1;j_1) $$  By the Cauchy inequality and (\ref{p2}), we get $$
\sum_{0<|j_1- n| \leq  n/2} \frac{r(m+j_1)}{|n-j_1|} \sigma_1
(n,1;j_1) \leq  \|r\|\cdot \sup_{0<|j_1- n| \leq  n/2} \sigma_1
(n,1;j_1) \leq \|r\| \tilde{\rho}_n, $$ and $$ \sum_{|j_1-n|>n/2}
\frac{r(m+j_1)}{|n-j_1|} \sigma_1 (n,1;j_1) \leq
\sum_{|j_1-n|>n/2} \frac{r(m+j_1)}{|n-j_1|} \cdot \|r\| \leq
\frac{2\|r\|}{\sqrt{n}} \cdot \|r\|. $$ Thus, in view of
(\ref{p1}), we have
\begin{equation}
\label{p6} \sigma_1 (n,2;m) \leq 2\|r\|\cdot \tilde{\rho}_n.
\end{equation}

On the other hand, for every  $s \in \mathbb{N},$ we have  $$
\sigma_1 (n,s+2;m) = \sum_{j_1,\ldots, j_s  \neq n}
\frac{r(m+j_1)}{|n-j_1|}\cdots \frac{r(j_{s-1}+j_s)}{|n-j_s|}
\sum_{j_{s+1},j_{s+2} \neq n} \frac{r(j_s +j_{s+1})}{|n-j_{s+1}|}
\frac{r(j_{s+1}+j_{s+2})}{|n-j_{s+2}|} $$ $$ = \sigma_1 (n,s;m)
\cdot \sup_{j_s} \sigma_1 (n,2;j_s).$$ Thus, by (\ref{p6}),
\begin{equation}
\label{p7} \sigma_1 (n,s+2;m)  \leq \sigma_1 (n,s;m) \cdot 2\|r\|
\tilde{\rho}_n.
\end{equation}
Now it is easy to see, by induction in $p,$
that  (\ref{p2}), (\ref{p6}) and (\ref{p7}) imply (\ref{p3}).
\end{proof}

\begin{Lemma}
\label{lemp3}
Let $r = (r(k)) \in \ell^2 (2\mathbb{Z}) $
be the sequence defined by (\ref{p26}),
and let
\begin{equation}
\label{p15} \sigma_2 (n,s;m) = \sum_{j_1, \ldots, j_s \neq   n}
r(m+j_1)\frac{r(j_1+j_2)}{|n+j_2|} \cdots
\frac{r(j_{s-2}+j_{s-1})}{|n+j_{s-1}|}
\frac{r(j_{s-1}+j_s)}{|n^2-j^2_s|}, \quad n \in \mathbb{N}, \;
s\geq 2.
\end{equation}
where $ m, j_1, \ldots, j_s \in n+ 2\mathbb{Z}. $ Then we have
\begin{equation}
\label{p16} \sigma_2 (n,2;m) \leq \|r\|^2 \cdot \frac{2\log \, 6n}{n}
\end{equation}
and
\begin{equation}
\label{p17} \sigma_2 (n,s;m) \leq \|r\|^2 \cdot \frac{2 \log \,
6n}{n} \cdot \sup_k \sigma_1 (n,s-2;k), \quad s \geq 3.
\end{equation}
\end{Lemma}

\begin{proof}
We have
\begin{equation}
\label{p18} \sigma_2 (n,2,m) = \sum_{j_2 \neq \pm n}
\frac{1}{|n^2-j_2^2 |} \sum_{j_1 \neq \pm n} r(m+j_1)r(j_1+j_2).
\end{equation}
By the Cauchy inequality, the sum over $ j_1 \neq \pm n  $ does not
exceed $\|r\|^2. $  Let us notice that
\begin{equation}
\label{p19} \sum_{j \neq \pm n} \frac{1}{n^2 - j^2} = \frac{2}{n}
\sum_1^{2n} \frac{1}{k} - \frac{1}{2n^2} < \frac{2 \log 6n}{n}.
\end{equation}
Therefore, (\ref{p18}) and (\ref{p19}) imply (\ref{p16}).

If $s\geq 3$ then  the sum $\sigma_2 (n,s;m) $ can be written in
the form

$$ \sigma_2 (n,s;m) = \sum_{j_s \neq \pm n} \frac{1}{|n^2-j^2_s|}
\sum_{j_2, \ldots, j_{s-1}\neq \pm n} \frac{r(j_2+j_3)}{|n+j_2|}
\cdots \frac{r(j_{s-1}+j_s)}{|n+j_{s-1}|} \sum_{j_1 \neq \pm n}
r(m+j_1)r(j_1+j_2).$$

Changing the sign of all indices, one can easily see that the
middle sum (over $j_2, \ldots, j_{s-1} $) equals $ \sigma_1
(n;s-2,j_s).$ Thus, we have

$$ \sigma_2 (n,s;m) \leq \sum_{j_s \neq \pm n}
\frac{1}{|n^2-j^2_s|} \sigma_1 (n;s-2,j_s) \cdot \sup_{j_2}
\sum_{j_1 \neq \pm n} r(m+j_1)r(j_1+j_2).$$

By the Cauchy inequality, the sum
over $ j_1 \neq \pm n  $
does not exceed $\|r\|^2. $

Therefore, by (\ref{p19}), we get (\ref{p17}).
\end{proof}

{\em Proof of Lemma \ref{lemp1}.}
We set
\begin{equation}
\label{p60} \varepsilon_n = M \cdot \left [ \left ( \frac{2 \log
6n}{n} \right )^{1/4} + (\tilde{\rho}_n)^{1/2} \right ],
\end{equation}
where $M = 4(1 + \|r\|)$  is chosen so that for large enough $n$
\begin{equation}
\label{p61} \sup_m \sigma_1 (n,2p,m) \leq (\varepsilon_n/2)^{2p},
\quad \sup_m \sigma_1 (n,2p+1,m) \leq \|r\|
(\varepsilon_n/2)^{2p}.
\end{equation}
Then, for large enough $n,$ we have
\begin{equation}
\label{p62} \sup_m \sigma_2 (n,s,m) \leq \frac{1}{M}
(\varepsilon_n/2)^{s+1}.
\end{equation}

Indeed, by the choice of $M,$ we have
\begin{equation}
\label{p62a}
 \|r\|^2 \cdot \frac{2\log \, 6n}{n}
  \leq  \|r\|^2 (\varepsilon_n/M)^4 \leq
\frac{1}{M^2} (\varepsilon_n/2)^4.
\end{equation}
Since $\varepsilon_n \to 0, $ there is $n_0 $ such that $
\varepsilon_n <1 \quad \text{for} \; n \geq n_0.$ Therefore, if
$n\geq n_0, $ then
 (\ref{p16}) and (\ref{p62a})
yields
 (\ref{p62}) for $s=2.$
If $s=2p $ with $p>1,$ then (\ref{p17}), (\ref{p62a}) and
(\ref{p61}) imply, for $n \geq n_0,$ $$\sup_m \sigma_2 (n,2p,m) \leq
\frac{1}{M^2} (\varepsilon_n/2)^4 \cdot (\varepsilon_n/2)^{2p-2}
\leq \frac{1}{M}(\varepsilon_n/2)^{2p+1}.$$

In an analogous way, for $n \geq n_0,$  we get $$\sup_m \sigma_2
(n,2p+1,m) \leq \frac{1}{M^2}(\varepsilon_n/2)^4 \cdot
\|r\|(\varepsilon_n/2)^{2(p-1)} \leq
\frac{1}{M}(\varepsilon_n/2)^{2p+2}, $$ which completes the proof of
(\ref{p62}).

Next we estimate $\sigma (n,s)$ by induction in $s.$ By
(\ref{p2}), we have for $n\geq n_0,$
\begin{equation}
\label{p63}
 \sigma (n,1) = \sum_{j_1 \neq \pm n} \frac{r(n-j_1)}{|n- j_1|} =
\sigma_1 (n,1;n) \leq \tilde{\rho}_n \leq (\varepsilon_n/2)^2 \leq
\varepsilon_n.
\end{equation}

For $s=2$ we get, in view of (\ref{p61}) and (\ref{p61}):

\begin{equation}
\label{p64}  \sigma (n,2) = \sum_{j_1, j_2 \neq \pm n}
\left (
\frac{1}{|n- j_1|} + \frac{1}{|n+ j_2|} \right ) \frac{1}{|n-
j_2|} r(n+j_1)r(j_1 + j_2)
\end{equation}
$$ \leq \sum_{j_1, j_2 \neq \pm n}  \frac{1}{|n- j_1|} \cdot
\frac{1}{|n- j_2|} r(n+j_1)r(j_1 + j_2) + \sum_{j_1, j_2 \neq \pm
n}  \frac{1}{|n+ j_2|} \cdot \frac{1}{|n- j_2|} r(n+j_1)r(j_1 +
j_2) $$ $$ = \sigma_1 (n,2,n) + \sigma_2 (n,2,n) \leq
(\varepsilon_n/2)^2 +(\varepsilon_n/2)^2 \leq (\varepsilon_n)^2.
$$

 Next we estimate $ \sigma (n,s), \;s\geq 2, $
Recall that $\sigma (n,s) $ is the sum of terms of the form $$ \Pi
(j_1, \ldots, j_s) r(n+j_1) r(j_1 +j_2 ) \cdots r(j_{s-1}+j_s), $$
where
\begin{equation}
\label{p65}
\Pi (j_1, \ldots, j_s) =
\left ( \frac{1}{|n- j_1|} + \frac{1}{|n+ j_2|} \right )
\cdots
\left ( \frac{1}{|n- j_{s-1}|} + \frac{1}{|n+ j_s|} \right )
\frac{1}{|n- j_s|}.
\end{equation}
By opening the parentheses we get
\begin{equation}
\label{p66} \Pi (j_1, \ldots, j_s) = \sum_{\delta_1,
\ldots,\delta_{s-1}=\pm 1} \left ( \prod_{\nu=1}^{s-1}
\frac{1}{|n+\delta_\nu j_{\nu +\tilde{\delta}_\nu}|} \right )
\frac{1}{|n-j_s|}, \quad \tilde{\delta}_\nu = \frac{1+\delta_\nu
}{2}.
\end{equation}
Therefore,
\begin{equation}
\label{p67} \sigma (n,s) = \sum_{\delta_1, \ldots,\delta_{s-1}=\pm
1} \tilde{\sigma}(\delta_1, \ldots,\delta_{s-1}),
\end{equation}
where
\begin{equation}
\label{p68} \tilde{\sigma}(\delta_1, \ldots,\delta_{s-1})=
\sum_{j_1, \ldots, j_s \neq \pm n} \left ( \prod_{\nu=1}^{s-1}
\frac{1}{|n+\delta_\nu j_{\nu +\tilde{\delta}_\nu}|} \right )
\frac{1}{|n-j_s|} r(n+j_1)r(j_1 +j_2)\cdots r(j_{s-1}+j_s).
\end{equation}

In view of (\ref{p52}), (\ref{p53}) and (\ref{p67}), Lemma
\ref{lemp1} will be proved if we show that
\begin{equation}
\label{p70} \tilde{\sigma}(\delta_1, \ldots,\delta_{s-1}) \leq
(\varepsilon_n/2)^s, \quad s\geq 2.
\end{equation}
We prove (\ref{p70}) by induction in $s.$

If $s=2$ then $$\tilde{\sigma}(-1)= \sigma_1 (n,2,n) \leq
(\varepsilon_n/2)^2, $$ and $$\tilde{\sigma}(+1)= \sigma_2 (n,2,n)
\leq (\varepsilon_n/2)^2. $$

If $s=3$ then there are four cases: $$\tilde{\sigma}(-1,-1)=
\sigma_1 (n,3,n) \leq (\varepsilon_n/2)^3; \quad
\tilde{\sigma}(+1,+1)= \sigma_2 (n,3,n) \leq (\varepsilon_n/2)^3;
$$ $$\tilde{\sigma}(-1,+1)= \sum_{j_1 \neq \pm n}
\frac{r(n+j_1)}{|n-j_1|} \sum_{j_2, j_3 \neq \pm n}
\frac{r(j_1+j_2)}{|n+j_3|}\frac{r(j_2+j_3)}{|n-j_3|}$$ $$=
\sum_{j_1 \neq \pm n} \frac{r(n+j_1)}{|n-j_1|} \sigma_2
(n,2,j_1)$$ $$ \leq \sigma_1 (n,1,n) \cdot \sup_m \sigma_2 (n,2,m)
\leq \|r\| \frac{1}{K} (\varepsilon_n/2)^3 \leq
(\varepsilon_n/2)^3;$$ $$\tilde{\sigma}(+1,-1)= \sum_{j_1,j_2 \neq
\pm n} \frac{r(n+j_1)r(j_1 +j_2)}{|n^2-j^2_2|} \sum_{j_3 \neq \pm
n} \frac{r(j_2 +j_3)}{|n-j_3|} $$ $$ \leq \sigma_2 (n,2,n) \cdot
\sup_m \sigma_1 (n,1,m) \leq \frac{1}{K}(\varepsilon_n/2)^3 \|r\|
\leq (\varepsilon_n/2)^3.$$

Next we prove that if (\ref{p70}) hold for some $s,$ then it holds
for $s+2.$ Indeed, let us consider the following cases:

(i)    $\delta_s =     \delta_{s+1} = -1;         $ then we have
$$ \tilde{\sigma}(\delta_1, \ldots,\delta_{s-1},-1,-1)
=\sum_{j_1,\ldots,j_s \neq \pm n} \left ( \prod_{\nu=1}^{s-1}
\frac{1}{|n+\delta_\nu j_{\nu +\tilde{\delta}_\nu}|} \right )
\frac{1}{|n-j_s|}$$ $$ \times \;  r(n+j_1)r(j_1 +j_2)\cdots
r(j_{s-1} +j_s) \sum_{j_{s+1}, j_{s+2} \neq \pm n} \frac{r(j_s
+j_{s+1})}{|n-j_{s+1}|}\frac{r(j_{s+1} +j_{s+2})}{|n-j_{s+2}|}$$
$$ =\sum_{j_1,\ldots,j_s \neq \pm n} \left ( \prod_{\nu=1}^{s-1}
\frac{1}{|n+\delta_\nu j_{\nu +\tilde{\delta}_\nu}|} \right )
\frac{1}{|n-j_s|} r(n+j_1)\cdots r(j_{s-1} +j_s) \sigma_1
(n,2,j_s) $$ $$ \leq \tilde{\sigma}(\delta_1, \ldots,\delta_{s-1})
\cdot \sup_m  \sigma_1 (n,2,m) \leq (\varepsilon_n/2)^s \cdot
(\varepsilon_n/2)^2 = (\varepsilon_n/2)^{s+2}. $$

(ii)   $\delta_s =-1,   \;     \delta_{s+1} = + 1;$ then we have
$$ \tilde{\sigma}(\delta_1, \ldots,\delta_{s-1},-1,+1)
=\sum_{j_1,\ldots,j_s \neq \pm n} \left ( \prod_{\nu=1}^{s-1}
\frac{1}{|n+\delta_\nu j_{\nu +\tilde{\delta}_\nu}|} \right )
\frac{1}{|n-j_s|}$$ $$ \times \;  r(n+j_1)r(j_1 +j_2)\cdots
r(j_{s-1} +j_s) \sum_{j_{s+1}, j_{s+2} \neq \pm n} \frac{r(j_s
+j_{s+1})r(j_{s+1} +j_{s+2})}{|n^2-j^2_{s+2}|}$$
 $$ \leq
\tilde{\sigma}(\delta_1, \ldots,\delta_{s-1}) \cdot \sup_m
\sigma_2 (n,2,m) \leq (\varepsilon_n/2)^s \cdot
(\varepsilon_n/2)^2 = (\varepsilon_n/2)^{s+2}. $$

(iii) $\delta_s = \delta_{s+1} = +1;$ then, if $\delta_1 = \cdots
= \delta_{s-1}= +1, $ we have $$ \tilde{\sigma}(\delta_1,
\ldots,\delta_{s+1})= \sigma_2 (n,s+2,n) \leq
(\varepsilon_n/2)^{s+2}. $$

Otherwise, let $\mu < s$  be the largest index such that
$\delta_\mu = -1.$ Then we have

$$ \tilde{\sigma}(\delta_1, \ldots,\delta_{s-1},+1,+1)
=\sum_{j_1,\ldots,j_\mu \neq \pm n} \left ( \prod_{\nu=1}^{\mu-1}
\frac{1}{|n+\delta_\nu j_{\nu +\tilde{\delta}_\nu}|} \right )
\frac{1}{|n-j_\mu|}$$ $$ \times \;  r(n+j_1)r(j_1 +j_2)\cdots
r(j_{\mu-1} +j_\mu) \sigma_2 (n,s+2-\mu,j_\mu)             $$ $$
\leq \tilde{\sigma}(\delta_1, \ldots,\delta_{\mu-1}) \cdot \sup_m
\sigma_2 (n,s+2-\mu,j_\mu) \leq (\varepsilon_n/2)^{\mu} \cdot
(\varepsilon_n/2)^{s+2-\mu} = (\varepsilon_n/2)^{s+2}. $$

(iv) $\delta_s =+1, \; \delta_{s+1} = -1;$ then, if $\delta_1 =
\cdots = \delta_{s-1}= +1, $ we have $$ \tilde{\sigma}(\delta_1,
\ldots,\delta_{s+1},-1)= \tilde{\sigma}(+1,\ldots, +1,-1, -1)=
$$$$ =\sum_{j_1,\ldots,j_{s+1}\neq \pm n} \left ( \prod_{\nu =1}^s
\frac{1}{|n+ j_{\nu +1}|} \right ) \frac{1}{|n-j_{s+1}|}
r(n+j_1)\cdots r(j_s +j_{s+1}) \sigma_1 (n,1, j_{s+1}) $$ $$ \leq
\sigma_2 (n,s+1,n) \cdot \sup_m \sigma_1 (n,1,m) \leq \frac{1}{K}
(\varepsilon_n/2)^{s+2} \cdot \|r\| \leq (\varepsilon_n/2)^{s+2}.
$$

Otherwise, let $\mu <s$  be the largest index such that
$\delta_\mu = -1,\; 1\leq \mu < n. $ Then we have

$$ \tilde{\sigma}(\delta_1, \ldots,\delta_{s-1},+1,-1)
=\sum_{j_1,\ldots,j_\mu \neq \pm n} \left ( \prod_{\nu=1}^{\mu-1}
\frac{1}{|n+\delta_\nu j_{\nu +\tilde{\delta}_\nu}|} \right )
\frac{1}{|n-j_\mu|}$$ $$ \times \; \sum_{j_{\mu+1}, \ldots,
j_{s+1} \neq \pm n} \frac{r(j_\mu +j_{\mu+1})}{|n+j_{\mu+2}|}
\cdots
 \frac{r(j_{s-1}+j_{s})}{|n+j_{s+1}|}
\frac{r(j_{s}+j_{s+1})}{|n-j_{s+1}|} \sum_{j_{s+2}\neq \pm n}
\frac{r(j_{s+1}+j_{s+2})}{|n-j_{s+2}|}$$
 $$ \leq \tilde{\sigma}(\delta_1,\ldots,\delta_{\mu-1})
 \cdot \sup_m  \sigma_2 (n,s+1- \mu,m) \cdot \sup_k \sigma_1 (n,1,k)$$
 $$
\leq (\varepsilon_n/2)^{\mu} \cdot \frac{1}{K}
(\varepsilon_n/2)^{s+2-\mu}  \|r\| \leq (\varepsilon_n/2)^{s+2}.
$$ Hence (\ref{p70}) holds for $s\geq 2.$

Now (\ref{p53}), (\ref{p67}) and (\ref{p70}) imply (\ref{p52}),
which completes the proof of Lemma \ref{lemp1}.
\end{proof}

Now we are ready to accomplish the proof of Theorem \ref{thm1}.

\section{Proof of the main theorem}

We need -- because we want to use (\ref{p24b}) -- to give estimates
of $A(n,s)$ from (\ref{p24c}), or (\ref{p29}). By (\ref{p33a}) and
(\ref{p33}), we reduce such estimates to analysis of quantities $A_j
(n,s), \; j=1, \ldots, 7.$

With $\rho_n \in (\ref{p34})$ and $\varepsilon_n \in (\ref{p60}), $
we set
\begin{equation}
\label{18b1} \kappa_n = \max\{ \rho_n, \varepsilon_n\}.
\end{equation}
Then, by Lemma \ref{lemp1} (and Corollary \ref{cor1}), i.e., by the
inequality (\ref{p520}), we have (in view of
(\ref{p37}),(\ref{p40}),(\ref{p42}) and (\ref{p44})--(\ref{p47}))
the following estimates for $A_j:$
$$A_1 \leq 4 \kappa_n^{s+1}, \qquad A_j \leq 2^{s+1} \cdot 2 \kappa_n^{s+1},
\quad j=2,4;$$
$$A_j \leq 2^{s+1} \sum_{\nu=1}^s \left (2 \kappa_n^\nu \cdot
\kappa_n^{s-\nu +1} \right ) = s 2^{s+2} \kappa_n^{s+1},\quad
j=3,5;$$
$$A_6 \leq 2^{s+1} \sum_{\nu=1}^s \left (\kappa_n^\nu \cdot
\kappa_n^{s-\nu +1} \right ) = s 2^{s+1} \kappa_n^{s+1};$$
$$A_7 \leq 2^{s} \sum_{1\leq \nu + \mu \leq s} \left (4 \kappa_n^\nu \cdot
\kappa^{\mu-\nu} \cdot \kappa_n^{s-\mu +1} \right ) = s(s-1) 2^{s+1}
\kappa_n^{s+1}.$$ In view of (\ref{p27}), (\ref{p60}) and
(\ref{p33}), these inequalities imply
$$ A(n,s) \leq  (2+s)^2  (2\kappa_n)^{s+1}.$$
Therefore, the right--hand side of (\ref{p24b}) does not exceed
$$ \sum_{s=0}^\infty A(n,s) \leq  (4\kappa_n)\sum_{s=0}^\infty (s+1)(s+2)
(2\kappa_n)^s =\frac{8\kappa_n}{(1-2\kappa_n)^3}.
$$
Therefore, if $\kappa_n < 1/4$ (which holds for $ n \geq N^* $ with
a proper choice of $N^*$), then $\sum_{s=0}^\infty A(n,s) \leq
64\kappa_n. $ Thus, by (\ref{p24b}) and the notations (\ref{p210}),
\begin{equation}
\label{18d} \| P_n -P_n^0 \|_{L^1 \to L^\infty} \leq \sum_{k,m}
|B_{km} (n)| \leq 64 \kappa_n, \quad n \geq N^*,
\end{equation}
 where $ \kappa_n
\in (\ref{18b1}).$

This completes the proof of Theorem \ref{thm1}. Of course,
Proposition \ref{prop1} follows because $ \|T\|_{L^2 \to L^2} \leq
\|T\|_{L^1 \to L^\infty} $ for any well defined operator $T.$
\end{proof}

\section{Miscellaneous}

1. Theorem \ref{thm1} (or Proposition \ref{prop1}) is an essential
step in the proof of our general statement (see an announcement in
\cite{DM17}, Thm. 9, or \cite{DM16}, Thm. 23), about the
relationship between the rate of decay of spectral gap sequences
(and deviations) and the smoothness of the potentials $v$ under the
{\em a priori } assumption that $v$ is a singular potential, i.e.,
that $v \in H^{-1}_{Per}. $ To use the information about the
deviations $\delta_n = |\mu_n - \frac{1}{2} (\lambda^+_n +
\lambda^-_n|,$ this is done in the framework of the scheme suggested
by the authors in \cite{DM5}. The concluding steps will be presented
in an upcoming paper, the third after \cite{DM17} and the present
one. However, Theorem \ref{thm1} is important outside this context
as well. We will mention now the most obvious
corollaries.\vspace{2mm}

2. The following theorem holds.

\begin{Theorem}
\label{thm7} In the above notations, the $L^p$-norms, $1\leq p \leq
\infty, $ on Riesz subspaces $E^N = Ran \,S_N, $ and $E_n = Ran
\,P_n, \; n \geq N,$  are uniformly equivalent; more precisely,
\begin{equation}
\label{a1} \|f\|_1  \leq \|f\||_\infty \leq C(N) \|f\|_1,  \quad
\forall f \in E^N,
\end{equation}
and
\begin{equation}
\label{a2} \|f\||_\infty \leq 3 \|f\|_1,  \quad  \forall f \in E_n,
\quad n \geq N^* (v),
\end{equation}
where
\begin{equation}
\label{a3} C(N) \leq 50 N \ln N.
\end{equation}
\end{Theorem}

\begin{proof}
By (\ref{p21}), if $N$ is large enough,
\begin{equation}
\label{a5} \|P_n - P_n^0\|_{L^1 \to L^\infty} \leq \frac{1}{2},
\quad n\geq N.
\end{equation}
If we are more careful when using (\ref{18b1}),(\ref{18d}),
(\ref{p34}) and (\ref{p60}), we may claim (\ref{a5}) for $N$ such
that
\begin{equation}
\label{a6} 2^{9}(1+\|r\|) \left ( \mathcal{E}_{\sqrt{N}} (r) +
\frac{2}{N^{1/4}} \left ( \|r\|^{1/2} + (\ln 6N)^{1/4} \right )
\right ) \leq \frac{1}{2}.
\end{equation}
If $f \in E_n, \, n \geq N,$ we have
\begin{equation}
\label{a6a} f = P_n f = (P_n - P_n^0 )f + P_n^0 f,
\end{equation}
where, for $bc = Per^\pm$
\begin{equation}
\label{a32} P_n^0 f = f_n e^{inx} + f_{-n} e^{-inx}, \quad f_k =
\frac{1}{\pi} \int_0^\pi f(x)e^{-ikx} dx,
\end{equation}
and, for $bc = Dir, $
\begin{equation}
\label{a33} P_n^0 f = 2g_n \sin nx, \quad g_n =
\frac{1}{\pi}\int_0^\pi f(x) \sin nx dx.
\end{equation}
In either case $\|P_n f \|_\infty  \leq 2 \|f\|_1, $ and therefore,
if $ \|f\|_1 \leq 1 $ we have
\begin{equation}
\label{a34} \|f\|_\infty \leq \|(P_n - P_n^0 )f\|_\infty +\| P_n^0
f\|_\infty  \leq 1/2 + 2 \leq 3.
\end{equation}
Remind that a projection
\begin{equation}
\label{a35} S_N = \frac{1}{2\pi i} \int_{\partial R_N}
(z-L_{bc})^{-1} dz,
\end{equation}
where, as in (5.40), \cite{DM16},
\begin{equation}
\label{a41} R_N = \{ z\in \mathbb{C}: \; -N < Re z < N^2 +N, \; |Im
z | < N,
\end{equation}
is finite--dimensional (see \cite{DM16}, (5.54), (5.56), (5.57) for
$\dim S_N $). Now we follow the inequalities proven in \cite{DM16}
to explain (\ref{a1}) and (\ref{a3}). Lemma 20, inequality (5.41) in
\cite{DM16}, states that
\begin{equation}
\label{a42} \sup \{\|K_\lambda V K_\lambda \|_{HS}: \; \lambda \not
\in R_N, \;  Re \lambda \leq N^2 - N \} \leq C \left ( \frac{(\log
N)^{1/2}}{N^{1/4}} \|q\| + \mathcal{E}_{4\sqrt{N}} (q) \right ).
\end{equation}
But by (\ref{a35})
\begin{equation}
\label{a43} S_N -S_N^0 = \frac{1}{2\pi i} \int_\Gamma K_\lambda
\sum_{m=1}^\infty (K_\lambda V K_\lambda)^m K_\lambda d\lambda,
\end{equation}
where we can choose $\Gamma$ to be the boundary $\partial \Pi $ of
the rectangle
\begin{equation}
\label{a44} \Pi (H)= \{z \in \mathbb{C}: \; -H \leq Re \,z \leq N^2
+ N, \; |Im \,z| \leq H \}, \quad H \geq N.
\end{equation}
Then by (\ref{a42}) and (\ref{a43}) the norm of the sum in the
integrand can be estimated by
\begin{equation}
\label{a45} \left \|\sum_1^\infty  \right \|_{2\to 2} \leq
\sum_1^\infty \|K_\lambda V K_\lambda\|^m_{HS} \leq 1, \quad \forall
\lambda \in
\partial \Pi (H)
\end{equation}
if (compare with (\ref{a6})) $ N \geq N^*(q)  $ and $N^* = N^*(q)$
is chosen to guarantee that
\begin{equation}
\label{a51} \text{``the right side in (\ref{a42})``} \leq 1/2  \; \;
\text{for} \; N \geq N^*.
\end{equation}
The additional factor $K_\lambda $ is a multiplier operator defined
by the sequence $\tilde{K} =\{1/\sqrt{\lambda -k^2}\},$
 so its norms $ \|K_\lambda : L^1 \to L^2\|$ and
 $ \;\|K_\lambda : L^1 \to L^2\|$
are estimated by $2\tilde{\kappa},$
 where
\begin{equation}
\label{a53}  \tilde{\kappa}= \|\tilde{K_\lambda} : \; \ell^\infty
\to \ell^2 \| = \|\tilde{K_\lambda} : \; \ell^2 \to \ell^1 \|=\sum_k
\frac{1}{|\lambda - k^2|}.
\end{equation}
Therefore, by (\ref{a45}) and (\ref{a53}),
\begin{equation}
\label{a54}  \alpha(\lambda):= \|K_\lambda \left (\sum_1^\infty
\cdots \right )K_\lambda : \; L^1 \to L^\infty \| \leq \sum_k
\frac{4}{|\lambda - k^2|}.
\end{equation}
By Lemma 18(a) in \cite{DM16} (or, Lemma 79(a) in \cite{DM15})
\begin{equation}
\label{a55}
 \sum_k \frac{1}{|n^2 - k^2|+b} \leq C_1 \frac{\log b}{\sqrt{b}} \quad
 \text{if} \;\; n \in \mathbb{N}, \; b \geq 2.
\end{equation}
(In what follows $C_j, \, j=1,2, \ldots$ are absolute constants;
$C_1 \leq 12.$) These inequalities are used to estimate the norm
$\alpha(\lambda)$ on the boundary $\partial \Pi (H) = \cup I_k (H),
\; k=1,2,3,4,$ where
$$
I_1 (H)= \{z: \; Re \, z = -H, \; |Im \, z| \leq H\} $$
$$
I_2 (H)= \{z: \;  -H \leq Re \, z \leq N^2 +N  , \; Im \, z = H\}
$$
$$
I_3 (H)= \{z: \; Re \, z = N^2 +N, \; |Im \, z| \leq H\} $$
$$ I_2 (H)= \{z: \;  -H \leq Re \, z \leq N^2 +N  , \; Im \, z = -
H\}
$$
Then we get
$$
\int_{I_1}\alpha(\lambda) |d\lambda| \leq C_2  \frac{\log H}{\sqrt
{H}} \cdot H,  $$
$$
\int_{I_k}\alpha(\lambda) |d\lambda| \leq C_3  \frac{\log H}{\sqrt
{H}}\cdot N^2, \quad k=2,4. $$
$$
\int_{I_3}\alpha(\lambda) |d\lambda| \leq C_4 \int_0^H \frac{\log
(N+y)}{\sqrt{N+y}} dy \leq C_5\sqrt{H} \log H. $$ If we put $H= N^2$
and sum up these inequalities we get by (\ref{a43})
\begin{equation}
\label{a71} \| S_N -S_N^0\|_{L^1 \to L^\infty} \leq C_6 N \log N,
\end{equation}
where $C_6 $ is an absolute constant $\leq 600.$

Now, as in (\ref{a6a}) and (\ref{a32}), let us notice that for $g
\in E^N$
\begin{equation}
\label{a72} g = S_N g = (S_N -S_N^0) g + S_N^0 g,
\end{equation}
where
\begin{equation}
\label{a73} S_N^0 g = \sum_{|k|\leq N} g_k e^{ikx}, \quad k
\;\text{even for} \;bc =Per^+,\;\;\text{odd for} \;bc =Per^-,
\end{equation}
and
\begin{equation}
\label{a74} S_N^0 g = 2\sum_{|k|\leq N} \tilde{g}_k \sin kx, \quad
bc = Dir,
\end{equation}
where
\begin{equation}
\label{a75} g_k = \frac{1}{\pi}\int_0^\pi g(x) e^{ikx}dx, \quad
\tilde{g}_k = \frac{1}{\pi}\int_0^\pi g(x) \sin kx dx.
\end{equation}
In either case
\begin{equation}
\label{a76} \|S_N^0 g\|_\infty \leq 2N \|g\|_1.
\end{equation}
Therefore, by (\ref{a71}) and (\ref{a76}), if $\|f\|_1 \leq 1$ we
have
\begin{equation}
\label{a81} \|f\|_\infty \leq C_6 N \log N + 2N \leq C_7 N \log N,
\quad N \geq N^* \in (\ref{a51}).
\end{equation}
Let us fix $N_0 \geq N^*, N_*,$ where $N_*$ is determined by
(\ref{a6}), i.e., (\ref{a6}) holds if $N \geq N_*.$ Then, by
(\ref{a81}),
\begin{equation}
\label{a82} \|S_{N_0}\|_{L^1 \to L^\infty} \leq C N_0 \log N_0,
\end{equation}
and for $N >N_0 $ we may improve the estimate in (\ref{a81}).
Indeed,
$$ S_N = (S_N - S_{N_0}) + S_{N_0}= S_{N_0} + \sum_{N_0 +1}^N P_k
$$
and, by (\ref{a5}) and (\ref{a32}),  $\|P_k\|_{L^1 \to L^\infty}
\leq 3.$ Therefore, by (\ref{a82}), $$\|S_N \|_{L^1 \to L^\infty}
\leq C N_0 \log N_0 + (N-N_0) \leq 3N + C N_0 \log N_0. $$
\end{proof}
\vspace{2mm}

3. Of course, any estimates of the kind
\begin{equation}
\label{a91} \|S_N - S_{N_0}\|_{L^1 \to L^\infty} \leq C(N)
\end{equation}
with $C_N \to \infty $ as $N \to \infty $ are weaker than the claim
\begin{equation}
\label{a92} \omega_N =\|S_N - S_{N_0}\|_{L^1 \to L^\infty} \to 0
\end{equation}
or even that $\omega_N$ is a bounded sequence. For real--valued
potentials $v \in H^{-1} $  and  $bc = Dir, $ (\ref{a92}) would
follow from Theorem 1 in \cite{Sa} if its proof given in \cite{Sa}
were valid. For complex--valued potentials $v \in H^{-1}, $ when the
system of eigenfunctions is not necessarily orthogonal the statement
of Theorem 1 in \cite{Sa} is false. Maybe it could be corrected if
the ''Fourier coefficients'' are chosen as
$$
c_k (f) = \langle f,w_k \rangle
$$
where the system $\{ w_k\}$ is bi--orthonormal with respect to $\{
u_k\}, $ i.e.,
$$ \langle  u_j, w_k \rangle = \delta_{jk} $$
(not the way as it is done in \cite{Sa}). But more serious
oversight, not just a technical misstep, seems to be a crucial
reference to \cite{SS03}, without specifying lines or statements in
\cite{SS03}, to claim something that cannot be found there. Namely,
the author of \cite{Sa} alleges that in \cite{SS03} the following
statement is proven. {\em Let $\{ y_k (x)\}$ be a normalized system
of eigenfunctions of the operator $$ L= - d^2/dx^2 + v, \quad v \in
H^{-1} ([0,\pi]),
$$
considered with Dirichlet boundary conditions. Then
\begin{equation}
\label{a111} y_k (x) = \sqrt{2} \sin kx + \psi_k (x),
\end{equation}
where
\begin{equation}
\label{a112} \sup_{[0,\pi]} |\psi_k (x) | \in \ell^2.
\end{equation}
(Two more sup--sequences coming from derivatives $y^\prime_k$ are
claimed to be in $\ell^2$ as well.)}

However, what one could find in \cite{SS03}, Theorem 2.7 and Theorem
3.13(iv),(v), is the claim
\begin{equation}
\label{a113} \sup_{[0,\pi]} \sum_k |\psi_k (x) |^2   < \infty.
\end{equation}
Of course, (\ref{a112}) implies (\ref{a113}) but if $\{\psi_k \}$ is
a sequence of $L^\infty$-functions then (\ref{a113}) does not imply
(\ref{a112}).

\end{document}